\documentclass[sn-mathphys-num]{sn-jnl}

\usepackage[utf8]{inputenc}
\usepackage[T1]{fontenc}
\usepackage{graphicx}%
\usepackage{multirow}%
\usepackage{amsmath,amssymb,amsfonts}%
\usepackage{amsthm}%
\usepackage{mathrsfs}%
\usepackage[title]{appendix}%
\usepackage{xcolor}%
\usepackage{textcomp}%
\usepackage{manyfoot}%
\usepackage{booktabs}%
\usepackage{algorithm}%
\usepackage{algorithmicx}%
\usepackage{algpseudocode}%
\usepackage{listings}%
\newcommand*{\Scale}[2][4]{\scalebox{#1}{$#2$}}%

\theoremstyle{thmstyleone}%
\newtheorem{theorem}{Theorem}
\theoremstyle{thmstyletwo}%
\newtheorem{remark}{Remark}%

\theoremstyle{thmstylethree}%

\usepackage{tabularray}
\usepackage{multirow}
\theoremstyle{plain}

\theoremstyle{definition}

\theoremstyle{remark}

\newtheorem{thm}{Theorem}[section]

 \newtheorem{lem}[thm]{Lemma}
 
 \theoremstyle{definition}
 
 \theoremstyle{remark}

 \numberwithin{equation}{section}
\newcommand{\ds}{\displaystyle}

\def\R{{\mathbb{R}}}
\def\N{{\mathbb{N}}}

\begin{document}

\title[Q-condition in elliptic equations]{The effect of Q-condition in elliptic equations involving Hardy potential and singular convection term}


\author*[1]{ \sur{F. ACHHOUD }}\email{achhoud.fst@uhp.ac.ma}

\author[2]{ \sur{A. BOUAJAJA }}\email{abdelkader.bouajaja@uhp.ac.ma}

\author[2]{ \sur{H. REDWANE}}\email{hicham.redwane@uhp.ac.ma }

\affil*[1]{\orgdiv{Faculty of Scinces and Technologies MISI Laboratory}, \orgname{Hassan First University of Settat}, \city{B.P. 577 Settat}, \postcode{26000},\country{ Morocco}}

\affil[2]{\orgdiv{Faculty of Economics and Management}, \orgname{Hassan First University}, \city{Settat}, \postcode{26000},  \country{Morocco}}


\abstract{Using  an approach by contradiction we prove the existence
and uniqueness of a weak solution to a  quasi-linear  elliptic boundary value problem
with singular convection term and Hardy Potential. Whose simplest model is
\begin{equation*}
\Scale[0.8]{\ds u \in W_0^{1,2}(\mathcal{O})\cap L^\infty(\mathcal{O}) :  -\Delta u=-\mathcal{A}\text{div}\left(\frac{x}{\vert x\vert^2}u\right)+\lambda \frac{u}{\vert x\vert^2}+f(x),}
\end{equation*}
where \(\mathcal{O}\) is a bounded open set in \(\mathbb{R}^N\), $\left(\mathcal{A},\lambda\right) \in \left(0, \infty\right)^2$ and \(f\in W^{-1,2}(\mathcal{O})\).
Additionally, by taking advantage of the  regularizing effect of the interaction between the
coefficient of the zero order term and the datum, we establish the existence, uniqueness and regularity of a weak solution to a  quasi-linear  boundary value problem whose simplest example is 
\begin{equation*}
\Scale[0.8]{\ds u \in W_0^{1,2}(\mathcal{O})\cap L^\infty(\mathcal{O}) :  -\Delta u +a(x)\vert u\vert^{p-2}u=-\mathcal{A}\text{div}\left(\frac{x}{\vert x\vert^2}u\right)+\lambda \frac{u}{\vert x\vert^2}+f(x),}
\end{equation*}
under suitable assumptions on $a$ and $f$.}




\maketitle

\section{Introduction}
In the first part of this paper, we deal with the following quasi-linear elliptic  boundary value problem
\begin{equation}\label{prb1}
\begin{aligned}
\begin{cases}
\ds -\text{div}(\mathcal{M}(x)\nabla u-u\mathcal{V}(x))=\lambda \frac{u}{\vert x\vert^2}+f(x)\quad &\text{in}\; \mathcal{O},\\
u=0 \quad &\text{on}\; \partial\mathcal{O}.
\end{cases}
\end{aligned}
\end{equation}
Here $\mathcal{O}\subset \mathbb{R}^{N}(N\geq3)$ is a bounded smooth set and containing the origin.\\
$\mathcal{M}: \mathcal{O} \longrightarrow \mathbb{R}^{N\times N}$ is a bounded measurable matrix, which satisfies the standard assumption:
\begin{equation}\label{c1}
\alpha\vert  \varsigma\vert  ^{2} \leq \mathcal{M}(x) \varsigma \cdot \varsigma, \quad\vert \mathcal{M}(x)\vert \leq \beta, \quad\text{a.e. $x \in \mathcal{O}$, $\forall \varsigma \in \mathbb{R}^{N}$},
\end{equation}
for $\alpha$ and $\beta$ are two positive constants.
 
Moreover, $\mathcal{V}(x)$ is a measurable vector field, which satisfies
\begin{equation}\label{c2}
\left\vert \mathcal{V}(x)\right\vert \leq \frac{\mathcal{A}}{\vert x\vert}, \quad \mathcal{A} \in \mathbb{R}^{+}.
\end{equation}
$\lambda>0$, $f(x)$ is a measurable function such that
\begin{equation}\label{c3}
f(x) \in W^{-1,2}(\mathcal{O}).
\end{equation}


Assuming $E \in\left(L^N(\mathcal{O})\right)^N,$ the existence of solutions $u$ to the problem
\begin{equation}
\begin{aligned}
\begin{cases}
\ds -\text{div}(\mathcal{M}(x)\nabla u-u\mathcal{V}(x))=f(x)\quad &\text{in}\; \mathcal{O}\\
u=0 \quad &\text{on}\; \partial\mathcal{O}
\end{cases}
\end{aligned}
\end{equation} 
was established in \cite{3}. More specifically, the following existence results were demonstrated
 
%
\begin{table}[ht]
\centering
\begin{tabular}[t]{c|c|cr|}
\cmidrule{2-4}\\
&$f \in L^{\frac{2N}{N+2}}(\mathcal{O})\subset W^{-1,2}_0(\mathcal{O})$&$f \in L^1(\mathcal{O})$&\\
\\
\hline
\\
$w\in W^{1,z}_0(\mathcal{O})$&$ z=2,$ &$z=q, \quad 1<q<\ds\frac{N}{N-1}$&\\
\\
\hline
\end{tabular}
\vspace{1cm}
 \end{table}%
If $E(x)$ satisfies \eqref{c2}, which is an assumption slightly weaker than $E \in\left(L^N(\mathcal{O})\right)^N$, the size of $\mathcal{A}$ plays an important role. In the paper \cite{4} the following existence results were proved

\begin{table}[ht]
\centering
\begin{tabular}[t]{c|c|cr|}
\cmidrule{2-4}\\
&$\begin{cases} f \in L^{\frac{2N}{N+2}}(\mathcal{O})\subset W^{-1,2}_0(\mathcal{O})\\\ds\mathcal{A}<\alpha\mathcal{H} \end{cases}$&$\begin{cases} f\in L^1(\mathcal{O})\\ \ds\mathcal{A}<\alpha(N-2)\end{cases}$&\\
\\
\hline
\\
$w\in W^{1,z}_0(\mathcal{O})$&$ z=2$&$z=q, \quad 1<q<\ds\frac{N}{N-1}$&\\
\\
\hline
\end{tabular}
\vspace{1cm}
\end{table}%
%
Let's recall the classical Hardy inequality. If $v$ belongs to $W_{0}^{1,2}(\mathcal{O})$, then
\begin{equation}\label{ha}
\ds\mathcal{H}^{2} \int_{\mathcal{O}} \frac{\vert  v\vert  ^{2}}{\vert  x\vert  ^{2}} \,dx\leq \int_{\mathcal{O}}\vert  \nabla v\vert  ^{2} \,dx,
\end{equation}
where $\mathcal{H}^{2}=\left(\frac{N-2}{2}\right)^{2}$ is optimal and is not achieved (See \cite{JI} for more details).

Given  that $f$ satisfies \eqref{c3}, it's well known (see \cite{JI}) that the following boundary problem
\begin{equation}\label{prb10}
\begin{aligned}
\begin{cases}
\ds -\text{div}(\mathcal{M}(x)\nabla u)=\lambda \frac{u}{\vert x\vert^2}+f(x)\quad &\text{in}\; \mathcal{O},\\
u=0 \quad &\text{on}\; \partial\mathcal{O},
\end{cases}
\end{aligned}
\end{equation}
has a unique solution $u\in W^{1,2}_0(\mathcal{O})$ for any $\lambda<\alpha \mathcal{H}^2$. More specifically, the authors establish the existence result by solving a minimization problem, where the solution 
$u$ is found as the minimum of the functional
\begin{equation*}
\mathcal{J}(u):=\frac{1}{2}\int_\mathcal{O} \nabla u\;dx -\frac{\lambda}{2}\int_\mathcal{O} \frac{u^2}{\vert u\vert^2}\;dx-\int_\mathcal{O} fu\;dx.
\end{equation*}
Moreover, they used also an approach that uses a density argument and  the compactness result given by  Boccardo and Murat (see \cite{LM}) to approximate the solution through a sequence of functions that converge within a suitable framework.

Now, we state our first result that will be proved in the Section 2 by using the same approach in the paper \cite{LSG}, in which the proof  based on establishing the main classical  a priori estimate by contradiction.
\begin{theorem}\label{th1}
 Assume that \eqref{c1}, \eqref{c2} and \eqref{c3} hold. Let $\mathcal{A}$ and $\lambda$ be positive constants that satisfy
\begin{equation}\label{H1}
  \alpha \mathcal{H}^2>\mathcal{A}\mathcal{H}+\lambda,
\end{equation}
then there exists a unique  weak solution for the problem \eqref{prb1}, that is a function  $u\in W^{1,2}_0(\mathcal{O})$  such that
\begin{equation}\label{wf}
\left\{\begin{array}{cr}
\ds u\mathcal{V}(x)\in (L^2(\mathcal{O}))^N,\quad\frac{u }{\vert x\vert ^{2}} \in L^{1}(\mathcal{O})\\
\ds\int_{\mathcal{O}} \mathcal{M}(x) \nabla u  \nabla \varphi\;dx=\ds\int_{\mathcal{O}} u\mathcal{V}(x) \nabla \varphi \;dx+\ds\lambda \int_{\mathcal{O}} \frac{u\varphi}{\vert x\vert ^{2}}\;dx+\ds\int_{\mathcal{O}} f(x) \varphi\;dx,
\end{array}\right.
\end{equation}
for every $\varphi$ in $W_{0}^{1,2}(\mathcal{O}) \cap L^{\infty}(\mathcal{O})$.
\end{theorem}

On the other hand, in \cite{12} it is proved that, since the problem \eqref{prb10} is linear, there is no solution for data in $L^1(\mathcal{O}).$ Furthermore, when  $f\in L^m(\mathcal{O})$ with $m\geq\frac{N}{2},$ even if $f\in L^\infty(\mathcal{O}),$ solutions to \eqref{prb10} are generally unbounded.

To overcome this failure of the existence of weak solution when the datum $f\in L^1(\mathcal{O})$, motivated by the recent papers by Arcoya \cite{11,ABO}, we will add the zero order term $\textit{a}(x)h(u)$ to the Problem \eqref{prb1}. Thus, we are dealing with  the following problem
\begin{equation}\label{prr}
\begin{aligned}
\begin{cases}
\ds -\text{div}(\mathcal{M}(x)\nabla u-u\mathcal{V}(x))+\textit{a}(x)h(u)=\lambda \frac{u}{\vert x\vert^2}+f(x)\quad &\text{in}\; \mathcal{O},\\
u=0 \quad &\text{on}\; \partial\mathcal{O}.
\end{cases}
\end{aligned}
\end{equation}
Here  $\textit{a}(x)$ is a measurable function such that
\begin{equation}\label{c33}
 0\leq \textit{a}(x)\in L^{1}(\mathcal{O}),
\end{equation}
and that
\begin{equation}\label{c44}
\vert  f(x)\vert   \leq Q \textit{a}(x), \quad \text{for every $Q>0$.}
\end{equation}
Moreover, we suppose that $h$ is a continuous, odd  and strictly increasing function satisfying
\begin{equation}\label{gcdd}
\begin{aligned}
\lim _{s \rightarrow \pm\infty} h(s)&= \pm\infty.
\end{aligned}
\end{equation}
We have to mention that in the paper \cite{11}, the authors proved  the regularizing effect of a polynomial lower order term $\textit{a}(x)u$  for the semi-linear problems.
In fact, they showed that, if $f\in L^1(\mathcal{O})$, it is possible to prove the existence, uniqueness  and the same summability obtained in \cite{12} of weak solutions by keeping the same assumptions on $\lambda$. 
Consequently, by using the regularizing effect,  we are going to prove that there exists a unique weak solution to the problem \eqref{prr}, which satisfies the same regularity as obtained in the papers \cite{12,4}.  
Before stating our result, we  introduce  the real function, which is a generalization of the function considered in \cite{ABO},  defined by
\begin{equation}\label{F}
\begin{aligned}
\mathcal{F}:\,[ 2^*,+\infty[&\longrightarrow \mathbb{R}\\
t \quad &\longmapsto \mathcal{F}(t):= \alpha \mathcal{H}\frac{N}{t}\left(2-\frac{2^*}{t}\right)+\mathcal{A}\mathcal{H}\left(\frac{2^*}{t}-1 \right),
\end{aligned}
\end{equation}
let us observe that the function $\mathcal{F}$ fulfills the following properties
\begin{equation}\label{Fhy}
\begin{cases}
\mathcal{F} \text{ is stricly decreasing, }\\
\mathcal{F}(2^*)=\alpha\mathcal{H}^2,\\
\lim_{t\rightarrow +\infty} \mathcal{F}(t)=\ds -\mathcal{A}\mathcal{H},\\
\exists !\, m_{\lambda,\mathcal{A}}>2^*,\, \text{such that}\, \mathcal{F}(m_{\lambda,\mathcal{A}})=\mathcal{A}\mathcal{H}+\lambda.
\end{cases}
\end{equation}
 \begin{theorem}\label{th2}
 Assume that \eqref{c1}, \eqref{c2} and $\eqref{c33}-\eqref{Fhy}$ hold. If
\begin{equation}\label{H1}
\alpha \mathcal{H}^2 >\mathcal{F}(m_{\lambda,\mathcal{A}}),
\end{equation}
then there exists a unique  weak solution $u\in W^{1,2}_0(\mathcal{O})\cap L^{\infty}(\mathcal{O})$ for the problem 
\begin{equation}\label{wf}
\left\{\begin{array}{cr}
\ds \textit{a}(x) h(u) \in L^{1}(\mathcal{O}),\quad \frac{u }{\vert x\vert ^{2}} \in L^{1}(\mathcal{O})\\
\ds\int_{\mathcal{O}} \mathcal{M}(x) \nabla u  \nabla \varphi\;dx+\ds\int_{\mathcal{O}} \textit{a}(x) h(u) \varphi\;dx=\ds\int_{\mathcal{O}} u\mathcal{V}(x) \nabla \varphi \;dx\\+\ds\lambda \int_{\mathcal{O}} \frac{u \varphi}{\vert x\vert ^{2}}\;dx+\ds\int_{\mathcal{O}} f(x) \varphi\;dx,
\end{array}\right.
\end{equation}
for every $\varphi$ in $W_{0}^{1,2}(\mathcal{O}) \cap L^{\infty}(\mathcal{O})$.

Moreover, $u$ satisfies the following regularity
%
$$u\in L^m(\mathcal{O})$$
for every $2^*\leq m<m_{\lambda,\mathcal{A}}.$
\end{theorem}
\begin{remark}
It is worth noting that in our paper, we establish the boundedness of the solution without relying on the domination of the potential term \( \frac{\lambda}{|x|^2} \) by the function \( a \), as assumed by the authors in \cite{9}.

\end{remark}
We stress that  the previous theorem gives the same  summability of the weak solution $u$ obtained separately in the papers \cite{12,4}. Indeed, in \cite{4} L. Boccardo studied the Dirichlet problem
\begin{equation*}
\begin{aligned}
\begin{cases}
\ds -\text{div}(\mathcal{M}(x)\nabla u)=-\text{div}(u\mathcal{V}(x))+f(x)\quad &\text{in}\; \mathcal{O},\\
u=0 \quad &\text{on}\; \partial\mathcal{O},
\end{cases}
\end{aligned}
\end{equation*}
where $\mathcal{V}$ satisfy \eqref{c2} and $f\in L^\rho(\mathcal{O}),$ with $1<\rho<\frac{N}{2}$. Particularly, he proved that
\begin{equation*}
u \in W_{0}^{1,2}(\mathcal{O}) \cap L^{\rho^{* *}}(\mathcal{O}) \quad \text{if $  \mathcal{A}  <\ds\frac{\alpha(N-2 \rho)}{\rho}$ and $\ds \frac{2 N}{N+2} < \rho$.}
\end{equation*}
 We explicitly observe that our theorem gives a weak solution $u \in W_{0}^{1,2}(\mathcal{O})$ such that
 \begin{equation*}
 u\in L^{m}(\mathcal{O}) \quad \text{if $\eqref{H1}$ is hold}\Leftrightarrow u\in L^{m}(\mathcal{O}) \quad \text{ if } \mathcal{A}\mathcal{H}<\mathcal{F}(m)
 \end{equation*}
 If this is the case, being $m=\rho^{**},$ we get
 \begin{align*}
 u\in L^{\rho^{**}}(\mathcal{O}) \quad \text{if $\mathcal{A}\mathcal{H}<\mathcal{F}(\rho^{**})$}&\Leftrightarrow u\in L^{\rho^{**}}(\mathcal{O}) \quad \text{if $\mathcal{A}\mathcal{H}<\mathcal{F}(\rho^{**})$},\\
 &\Leftrightarrow u\in L^{\rho^{**}}(\mathcal{O}) \quad \text{if $\mathcal{A}\frac{N(\rho-1)}{\rho}<\alpha\frac{N(\rho-1)(N-2\rho)}{\rho^2}$},\\
 &\Leftrightarrow u\in L^{\rho^{**}}(\mathcal{O}) \quad \text{if $\mathcal{A}<\ds\frac{\alpha(N-2 \rho)}{\rho}$},\\
 \end{align*}
Consequently, the regularity obtained in our paper is the same as in \cite{4}, by keeping the same condition on the size of $\mathcal{A}$ and take the datum $f$ in $L^1(\mathcal{O})$ which explain the contribution of the term $\textit{a}(x)h(u)$ on the problem \eqref{prb1}.

At least in the case $\mathcal{A}=0$ the existence and regularity of a weak solution $u\in W_{0}^{1,2}(\mathcal{O}) \cap L^{\rho^{* *}}(\mathcal{O})$ of the following problem
\begin{equation*}
\begin{aligned}
\begin{cases}
\ds -\text{div}(\mathcal{M}(x)\nabla u)=\lambda \frac{u}{\vert x\vert^2}+ f(x)\quad &\text{in}\; \mathcal{O},\\
u=0 \quad &\text{on}\; \partial\mathcal{O},
\end{cases}
\end{aligned}
\end{equation*}
 can be proved under a suitable assumption on $\lambda$ and $f\in L^\rho(\mathcal{O})$ with $1<\rho<\frac{N}{2}$ . More precisely, the authors in \cite{12}  proved that 
 \begin{equation*}
u \in W_{0}^{1,2}(\mathcal{O}) \cap L^{\rho^{* *}}(\mathcal{O}) \quad \text{if $  \lambda  <\ds\alpha\frac{N(\rho-1)(N-2 \rho)}{\rho^2}$ and $\ds \rho\geq\frac{2 N}{N+2}$.}
\end{equation*}
 In order to show that our theorem gives the same regularity as in the paper \cite{12}, remaining the same condition on $\lambda$ and with right hand side $f$ only in $L^1(\mathcal{O})$, we can repeat line by line the above proof using  \eqref{H1} and \eqref{Fhy}. Obviously, we ensue that
 \begin{align*}
 u\in L^{\rho^{**}}(\mathcal{O}) \quad \text{if $\lambda <\mathcal{F}(\rho^{**})$}&\Leftrightarrow  u\in L^{\rho^{**}}(\mathcal{O}) \quad \text{if $\lambda<\alpha\ds\frac{N(\rho-1)(N-2 \rho)}{\rho^2}$}.
 \end{align*}
Despite the presence of  lower order terms in our problem that may undermine the maximum principal (see \cite{30,22}), our second goal is to establish two results that guarantee that every solution of the problem \eqref{prr} satisfies the weak and strong maximum principal. We have to mention that the study of the maximum principal for the Problem \eqref{prr}, with $\lambda=0,$ $\mathcal{V}(x)\in L^{N}(\mathcal{O})$ and  regular datum $f\in L^m,\;m>1$, was initiated in \cite{10}. In fact the author proved  that if we assume that the data are greater
or equal than zero, but not identically zero, then the solution satisfies the strong maximum principal.
We point out that, the presence of the singular term in the right-hand side and the singular convection term in the left hand side forced us to consider a different approach, that used in \cite{11}, in order to prove that the solution satisfies the weak maximum principal. 

 Our result, which concerns the weak maximum principal, is the following.
\begin{theorem}(Weak maximum principle)\label{wmp}
Assume that \eqref{c1}, \eqref{c2} and $\eqref{c33}-\eqref{Fhy}$ hold.  If $f\geq 0$ and $\textit{a}(x)\geq 0$ are such that
\begin{equation}\label{hy}
f(x)=Q\textit{a}(x)\; \text{with $Q>0$},
\end{equation} 
and
\begin{equation}\label{hky1}
\max_{\vert s\vert\leq k_0}h(s)\leq Q.
\end{equation} 
then the weak solution $u \in W_{0}^{1,2}(\mathcal{O}) \cap L^{\infty}(\mathcal{O})$ given by theorem \ref{th2}  is such that $u \geq 0$
almost everywhere in $\mathcal{O}$.
\end{theorem}
In the next theorem we study the positivity, up to a zero measure set, of the solution of the Dirichlet Problem \eqref{prr}, the proof hinges on the approach of \cite{11}.
\begin{theorem}(Strong maximum principle)\label{smp}
Assume that \eqref{c1}, \eqref{c2} and $\eqref{c33}-\eqref{Fhy}$ hold. Let $u \in W_{0}^{1,2}(\mathcal{O}) \cap L^{\infty}(\mathcal{O})$ be the solution of \eqref{prr} given by Theorem \ref{th2}. If $f \geq 0$ (and not almost everywhere equal to zero), then for every set $\omega \subset \subset \mathcal{O}$, there exists $C_{\omega}>0$ such that $u(x) \geq C_{\omega}$ almost everywhere in $\omega$.
\end{theorem}

For ease of reading, we address the case of regular data in Section 2. In Section 3, we provide the proofs of the main results for the case of irregular data. Additionally, we present an example that illustrates our findings. We also demonstrate that if the data \( f \) is non-negative (and not identically zero), then both the weak and strong maximum principles hold for equation \eqref{prr}.
 \section{Regular data: $f\in W^{-1,p^\prime}(\mathcal{O})$}
 \subsection{A priori estimates}
 For $k>0$, and $t$ in $\R$, let us define
$$
T_k(t)=\max (-k, \min (t, k)), \quad G_k(t)=t-T_k(t)=(|t|-k)^{+} \operatorname{sgn}(t),
$$
As in \cite{LSG}, we consider the following approximate Dirichlet problems
 \small{
\begin{equation*}\label{aprpp}
\begin{aligned}
\begin{cases}
\ds -\text{div}\left(\mathcal{M}(x)\nabla u_n-\frac{u_n}{1+\frac{1}{n}\vert u_n\vert }\mathcal{V}_n(x)\right)=\ds \lambda\frac{u_n}{\vert x\vert^2+\frac{1}{n}}+f_n(x)\; &\text{in}\; \mathcal{O},\\
u_n=0 \; &\text{in}\; \partial\mathcal{O},
\end{cases}
\end{aligned}
\end{equation*}}
where
\begin{equation*}\label{apc1}
 f_{n}(x)=\frac{f(x)}{1+\frac{1}{n}\vert f(x)\vert }\quad \text{and}\quad \mathcal{V}_n(x)=\frac{\mathcal{V}(x)}{1+\frac{1}{n}\vert \mathcal{V}(x)\vert }.
\end{equation*}
The Schauder theorem enables us to ensure the existence of a weak solution, i.e. $u_{n}\in W_{0}^{1,2}(\mathcal{O})$, for the approximate problem. Furthermore, this solution satisfying
\small{
\begin{equation}\label{wfa}
\begin{aligned}
\ds\int_{\mathcal{O}} \mathcal{M}(x) \nabla u_n  \nabla \varphi\;dx&=\int_{\mathcal{O}} \frac{u_n}{1+\frac{1}{n}\vert u_n\vert }\mathcal{V}_n(x) \nabla \varphi \;dx+\lambda \int_{\mathcal{O}} \frac{u_n}{\vert x\vert ^{2}+\frac{1}{n}}\varphi\;dx\\&+\int_{\mathcal{O}} f_n(x) \varphi\;dx,
\end{aligned}
\end{equation}}
for every $\varphi$ in $W_{0}^{1,2}(\mathcal{O}) \cap L^{\infty}(\mathcal{O})$.

 Moreover, since for every fixed $n$ the functions $u_n,$ $\mathcal{V}_n$ and $f_n$ are bounded. Then $u_n$ belongs to $L^\infty(\mathcal{O})$, thanks to Stampacchia's boundedness theorem (see \cite{2}).
 \begin{lem}
 Let us assume \eqref{c1}, \eqref{c2} and $f \in W^{-1,2}(\mathcal{O})$. If $\mathcal{A}$ and $\lambda$ are such that
 \begin{equation}\label{Cp}
 \alpha \mathcal{H}^2>\mathcal{A}\mathcal{H}+\lambda.
 \end{equation}
 Then, there exists a positive constant $\mathcal{C}_1 $ independent of $n$, such that
 \begin{equation}\label{Cpp}
 \Vert u_n\Vert_{W^{1,2}_0(\mathcal{O})}\leq \mathcal{C}_1\Vert f\Vert_{W^{-1,2}(\mathcal{O})}, \quad\quad \forall n\in\N.
 \end{equation}
 \end{lem}
 \begin{proof}
 By contradiction, let us assume, for $m\in \N,$ the existence of a sequence $\left\{f_m\right\} \subset W^{-1,2}(\mathcal{O})$ and a sequence $\left\{n_m\right\} \subset \mathbb{N}$ such that
\begin{equation}\label{cc}
\ds\frac{\left\|f_m\right\|_{W^{-1,2}(\mathcal{O})}}{\left\|u_{n_m}\right\|_{W_0^{1,2}(\mathcal{O})}} < \frac{1}{m}.
\end{equation}
For brevity, we set $u_m = u_{n_m}$. Let us define
$$a_m = \frac{1}{1+\frac{1}{n_m}\left|u_m\right|},$$
and consider the normalized function defined by
\begin{equation}
w_m = \frac{u_m}{\left\|u_m\right\|_{W_0^{1,2}(\mathcal{O})}}.
\end{equation}
Such that
\begin{equation}\label{f1}
\begin{aligned}
\int_{\mathcal{O}} M(x) \nabla w_m \nabla \varphi \;dx&= \int_{\mathcal{O}} a_m w_m \mathcal{V}_{n_m}(x) \nabla \varphi \;dx +\lambda \int_{\mathcal{O}}  \frac{ w_m\varphi}{\vert x\vert^2+\frac{1}{n_m}}\;dx \\&+\frac{\left\langle f_m, \varphi\right\rangle}{\left\|u_m\right\|_{W_0^{1,2}(\mathcal{O})}} ,
\end{aligned}
\end{equation}
for every $\varphi \in W_0^{1,2}(\mathcal{O})\cap L^\infty(\mathcal{O})$.

 Since the sequence $\left\{w_m\right\}$ is bounded in $W_0^{1,2}(\mathcal{O})$. Then, up to a subsequence one has
\begin{align}
&w_m\rightarrow w\text{ a.e. in }\mathcal{O},\label{cvuae}\\
&w_m\rightharpoonup w \text{ weakly in $W^{1,2}_0(\mathcal{O})$},\label{cvuf}\\
&w_m\rightarrow w\text{ strongly in $L^{t}(\mathcal{O}),$ with $1\leq t<2^*.$}\label{cvut}
\end{align} 

Now, taking $T_k(w)$ as a test function in the weak formulation \eqref{f1}, we have
\begin{equation*}
\begin{aligned}
\int_{\mathcal{O}} M(x) \nabla w_m \nabla T_k(w)\;dx &= \int_{\mathcal{O}} a_m w_m \mathcal{V}_{n_m}(x) \nabla T_k(w)\;dx +\lambda \int_{\mathcal{O}}  \frac{ w_mT_k(w)}{\vert x\vert^2+\frac{1}{n_m}}\;dx\\& +\frac{\left\langle f_m, T_k(w)\right\rangle}{\left\|u_m\right\|_{W_0^{1,2}(\mathcal{O})}}.
\end{aligned}
\end{equation*}
Consequently, by \eqref{c2} and the fact that $\vert a_m\vert\leq 1$,
\begin{equation*}
\begin{aligned}
\int_{\mathcal{O}} M(x) \nabla w_m \nabla T_k(w)\;dx &\leq \mathcal{A}\int_{\mathcal{O}} \frac{\vert w_m\vert}{\vert x\vert}\vert \nabla T_k(w)\vert\;dx +\lambda \int_{\mathcal{O}}  \frac{\vert w_m\vert \vert T_k(w)\vert}{\vert x\vert^2}\;dx\\& +\frac{\left\langle f_m, T_k(w)\right\rangle}{\left\|u_m\right\|_{W_0^{1,2}(\mathcal{O})}}.
\end{aligned}
\end{equation*}
Since $\ds \vert w_m\vert\leq \vert T_k(w_m)\vert+\vert G_k(w_m)\vert$, we obtain
\begin{equation*}
\begin{aligned}
\int_{\mathcal{O}} M(x) \nabla w_m \nabla T_k(w)\;dx &\leq \mathcal{A}\int_{\mathcal{O}} \frac{\vert T_k(w_m)\vert}{\vert x\vert}\vert \nabla T_k(w)\vert\;dx+\mathcal{A}\int_{\mathcal{O}} \frac{\vert G_k(w_m)\vert}{\vert x\vert}\vert \nabla T_k(w)\vert\;dx \\&+\lambda \int_{\mathcal{O}}  \frac{\vert T_k(w_m)\vert \vert T_k(w)\vert}{\vert x\vert^2}\;dx+\lambda \int_{\mathcal{O}}  \frac{\vert G_k(w_m)\vert \vert T_k(w)\vert}{\vert x\vert^2}\;dx\\& +\frac{\left\langle f_m, T_k(w)\right\rangle}{\left\|u_m\right\|_{W_0^{1,2}(\mathcal{O})}}.
\end{aligned}
\end{equation*}
By Holder and Hardy's inequalities, we have
\begin{equation*}
\begin{cases}
\ds\mathcal{A}\int_{\mathcal{O}} \frac{\vert G_k(w_m)\vert}{\vert x\vert}\vert \nabla T_k(w)\vert\;dx \leq \frac{\mathcal{A}}{\mathcal{H}}\Vert w_m\Vert_{W^{1,2}_0(\mathcal{O})}\left( \int_{\mathcal{O}} \vert \nabla T_k(w)\vert^2\chi_{\lbrace\vert w_m\vert>k\rbrace}\;dx \right)^{\frac{1}{2}},\\
\\
\ds\lambda \int_{\mathcal{O}}  \frac{\vert G_k(w_m)\vert \vert T_k(w)\vert}{\vert x\vert^2}\;dx\leq \frac{\lambda}{\mathcal{H}^2}\Vert w_m\Vert_{W^{1,2}_0(\mathcal{O})}\left( \int_{\mathcal{O}} \vert \nabla T_k(w)\vert^2\chi_{\lbrace\vert w_m\vert>k\rbrace}\;dx \right)^{\frac{1}{2}},
\end{cases}
\end{equation*}
since $\Vert w_m\Vert_{W^{1,2}_0(\mathcal{O})}=1$ and 
$$ \vert \nabla T_k(w)\vert^2\chi_{\lbrace\vert w_m\vert>k\rbrace}\rightarrow 0 \text{ strongly in $L^2(\mathcal{O})$},$$
it follows that
\begin{equation*}
\begin{cases}
\ds\lim_{m\rightarrow +\infty}\mathcal{A}\int_{\mathcal{O}} \frac{\vert G_k(w_m)\vert}{\vert x\vert}\vert \nabla T_k(w)\vert\;dx=0,\\
\\
\ds\lim_{m\rightarrow +\infty} \lambda \int_{\mathcal{O}}  \frac{\vert G_k(w_m)\vert \vert T_k(w)\vert}{\vert x\vert^2}\;dx=0.
\end{cases}
\end{equation*}
Moreover, thanks to \eqref{cvuf} and \eqref{cc}, we obtain
\begin{equation*}
\begin{cases}
\ds\lim_{m\rightarrow +\infty}\int_{\mathcal{O}} M(x) \nabla w_m \nabla T_k(w)\;dx=\int_{\mathcal{O}} M(x) \nabla T_k(w) \nabla T_k(w)\;dx,\\
\\
\ds\lim_{m\rightarrow +\infty} \frac{\left\langle f_m, T_k(w)\right\rangle}{\left\|u_m\right\|_{W_0^{1,2}(\mathcal{O})}}=0.
\end{cases}
\end{equation*}
In addition, by \eqref{cvuae}, we get
\begin{equation*}
\begin{cases}
\ds\lim_{m\rightarrow +\infty}\mathcal{A}\int_{\mathcal{O}} \frac{\vert T_k(w_m)\vert}{\vert x\vert}\vert \nabla T_k(w)\vert\;dx=\mathcal{A}\int_{\mathcal{O}} \frac{\vert T_k(w)\vert}{\vert x\vert}\vert \nabla T_k(w)\vert\;dx,\\
\\
\ds\lim_{m\rightarrow +\infty} \lambda \int_{\mathcal{O}}  \frac{\vert T_k(w_m)\vert \vert T_k(w)\vert}{\vert x\vert^2}\;dx=\lambda \int_{\mathcal{O}}  \frac{\vert T_k(w)\vert^2}{\vert x\vert^2}\;dx.
\end{cases}
\end{equation*}
Therefore, we deduce that
\begin{equation*}
\begin{aligned}
\int_{\mathcal{O}} M(x) \nabla w \nabla T_k(w)\;dx &\leq \mathcal{A}\int_{\mathcal{O}} \frac{\vert T_k(w)\vert}{\vert x\vert}\vert \nabla T_k(w)\vert\;dx+\lambda \int_{\mathcal{O}}  \frac{\vert T_k(w)\vert^2}{\vert x\vert^2}\;dx.
\end{aligned}
\end{equation*}
Thus, applying Young's, Hardy's and Poincare inequalities, we derive that
 \begin{equation}
 \ds\mu_1 \left(\alpha-\frac{\mathcal{A}}{\mathcal{H}} -\frac{\lambda}{\mathcal{H}^2}\right)\int_\mathcal{O} \vert T_k(w)\vert^2\;dx\leq \frac{\mathcal{A} k^2 \mathcal{H}}{4}\int_{\lbrace\vert w\vert\leq k \rbrace} \frac{dx}{\vert x\vert^2},
 \end{equation}
 where $\mu_1$ is the minimal eigenvalue of the negative Laplacian.
 
 Using now that for any $k<h$, by following the papers \cite{3,130}, 
 \begin{equation}
 \ds\mu_1 \left(\alpha-\frac{\mathcal{A}}{\mathcal{H}} -\frac{\lambda}{\mathcal{H}^2}\right) \left\vert\lbrace \vert w\vert>h\rbrace \right\vert\leq \frac{\mathcal{A}\mathcal{H}}{4}\int_{\lbrace\vert w\vert\leq k \rbrace} \frac{dx}{\vert x\vert^2};
 \end{equation}
 Hence, by \eqref{Cp}, we obtain
 \begin{equation}\label{est1}
 \ds \left\vert\lbrace \vert w\vert>h\rbrace \right\vert\leq \frac{\mathcal{A}}{4\mu_1 \left(\alpha-\frac{\mathcal{A}}{\mathcal{H}} -\frac{\lambda}{\mathcal{H}^2}\right)}\lim_{k\rightarrow 0}\int_{\lbrace\vert w\vert\leq k \rbrace} \frac{dx}{\vert x\vert^2};
 \end{equation}
 i.e.,
 \begin{equation}
 \left\vert\lbrace \vert w\vert>h\rbrace \right\vert= 0,\quad\quad \forall h>0
 \end{equation}
 which means that $w\equiv 0.$ In consequence, 
 \begin{align}
&w_m\rightarrow 0\text{ a.e. in }\mathcal{O},\label{cvuae1}\\
&w_m\rightharpoonup 0 \text{ weakly in $W^{1,2}_0(\mathcal{O})$}.\label{cvuf1}
\end{align} 
To complete the proof, we must show that
\begin{equation}
w_m\rightarrow 0 \text{ strongly in $W^{1,2}_0(\mathcal{O})$}.
\end{equation}
To this aim, we choose $T_k(w_m)$ as test function in \eqref{f1}, we have 
\begin{equation*}
\begin{aligned}
\int_{\mathcal{O}} M(x) \nabla w_m \nabla T_k(w_m)\;dx &= \int_{\mathcal{O}} a_m w_m \mathcal{V}_{n_m}(x) \nabla T_k(w_m)\;dx +\lambda \int_{\mathcal{O}}  \frac{w_mT_k(w_m)}{\vert x\vert^2+\frac{1}{n_m}}\;dx\\& +\frac{\left\langle f_m, T_k(w_m))\right\rangle}{\left\|u_m\right\|_{W_0^{1,2}(\mathcal{O})}}.
\end{aligned}
\end{equation*}
Using \eqref{c1}, \eqref{c2}, \eqref{cc} and recalling that $\ds\vert a_m(x)\vert\leq 1$ we obtain
\begin{equation*}
\begin{aligned}
\alpha\int_{\mathcal{O}} \vert\nabla T_k(w_m)\vert^2\;dx &\leq\mathcal{A} \int_{\mathcal{O}}\frac{\vert T_k(w_m)\vert}{\vert x\vert}   \vert\nabla T_k(w_m)\vert\;dx +\lambda \int_{\mathcal{O}}  \frac{ \vert w_m\vert \vert T_k(w_m)\vert}{\vert x\vert^2}\;dx\\& +\frac{1}{m}.
\end{aligned}
\end{equation*}
Applying Holder, Hardy's inequalities and using that $\vert w_m\vert\leq \vert T_k(w_m)\vert+\vert G_k(w_m)\vert$ we get
\begin{equation*}
\begin{aligned}
\left(\alpha-\frac{\mathcal{A}}{\mathcal{H}}-\frac{\lambda}{\mathcal{H}^2} \right)\int_{\mathcal{O}} \vert\nabla T_k(w_m)\vert^2\;dx &\leq\lambda \int_{\mathcal{O}}  \vert g_k(x)\vert\;dx +\frac{1}{m}.
\end{aligned}
\end{equation*}
with
$$
g_k(x)=\lambda \frac{T_k(w_m(x)) G_k(w_m(x))}{|x|^2} .
$$
Thanks to \eqref{cvuae1} we have that
\begin{align*}
g_k\rightarrow 0\text{ a.e. in }\mathcal{O},
\end{align*} 
and since
$$
\vert g_k(x)\vert\leq \lambda \frac{\vert w_m(x)\vert^2}{|x|^2} \in L^1(\mathcal{O}),
$$
then, by Lebesgue theorem, we deduce that 
\begin{align*}
&g_k\rightarrow 0 \text{ strongly in $L^{1}(\mathcal{O})$}.
\end{align*} 
taking the limit as $k$ tends to infinity, it follows
\begin{equation*}
\begin{aligned}
\left(\alpha-\frac{\mathcal{A}}{\mathcal{H}}-\frac{\lambda}{\mathcal{H}^2} \right)\int_{\mathcal{O}} \vert\nabla w_m\vert^2\;dx &\leq\frac{1}{m}.
\end{aligned}
\end{equation*}
At last, by recalling \eqref{Cp}, we get
 \begin{equation*}
\begin{aligned}
\limsup_{m\rightarrow +\infty}\int_{\mathcal{O}} \vert\nabla w_m\vert^2\;dx = 0.
\end{aligned}
\end{equation*}
This conflicts with the fact that $\Vert u_m \Vert_{W^{1,2}_0(\mathcal{O})}=1$. Thus, the estimates \eqref{Cpp} holds. 
 \end{proof}
 \subsection{Proof of Theorem \ref{th1}.}
In the previous lemma we have proved the boundedness of the sequence $\lbrace u_n \rbrace$ in $W^{1,2}_0(\mathcal{O})$. Therefore, up to subsequence, there exists a function $u\in W^{1,2}_0(\mathcal{O})$ such that 
\begin{align*}
&u_n\rightarrow u\text{ a.e. in }\mathcal{O},\\
&u_n\rightharpoonup u \text{ weakly in $W^{1,2}_0(\mathcal{O})$},\\
&u_n\rightarrow u \text{ strongly in $L^{2}(\mathcal{O})$}.
\end{align*}
Moreover, thanks to Hardy's inequality, we have that $$u_n \mathcal{V}_n(x)\in (L^2(\mathcal{O}))^2.$$ Then, we obtain that
\begin{equation*}
\ds u_n \mathcal{V}_n(x)\rightarrow u \mathcal{V}(x) \text{ strongly in $(L^{2}(\mathcal{O}))^2$}.
\end{equation*}
In addition, we have that
\begin{equation*}
\ds\frac{u_n}{\vert x\vert^2+\frac{1}{n}}  \rightarrow \frac{u}{\vert x\vert^2} \text{ strongly in $L^{1}(\mathcal{O})$}.
\end{equation*}
Hence we can pass to the limit in \eqref{wfa},  thanks to the linearity of the problem, to prove that $u$ is a weak solution of the Problem \eqref{prb1} in the sens of \eqref{wf}.

Now, Let's consider \(u\) and \(v\) belonging to \(W_0^{1,2}(\mathcal{O})\) as two solutions of the Problem \eqref{prb1}. Consequently, for every $\varphi \in W_0^{1,2}(\mathcal{O})$, \(w=u-v\) satisfies the following equation

\begin{equation}
\int_{\mathcal{O}} M(x) \nabla w \nabla \varphi \;dx =\int_{\mathcal{O}} w \mathcal{V}(x) \nabla \varphi \;dx+\lambda \int_\mathcal{O} \frac{w \varphi }{ \vert x\vert^2}\;dx.
\end{equation}
By using \(T_h(w)\) as a test function in the aforementioned equation, we derive the identical estimate as presented in \eqref{est1}. Thus, by following the same logic as in the preceding lemma, we deduce

\begin{equation}
|\{|w|>h\}| = 0, \quad \text { for any } \quad h>0,
\end{equation}
indicating that \(w \equiv 0\), meaning \(u \equiv v\).

\section{No regular data: $f\in L^1(\mathcal{O})$ }
\subsection{A priori estimates}


As in \cite{ABO}, we consider the following approximate Dirichlet problems
 \small{
\begin{equation*}\label{aprpp}
\begin{aligned}
\begin{cases}
\ds -\text{div}(\mathcal{M}(x)\nabla u_n-T_n(u_n)\mathcal{V}_n(x))+a_n(x)h(u_n)=\ds\lambda \frac{T_n(u_n)}{\vert x\vert^2+\frac{1}{n}}+f_n(x)\; &\text{in}\; \mathcal{O},\\
u_n=0 \; &\text{in}\; \partial\mathcal{O},
\end{cases}
\end{aligned}
\end{equation*}}
where
\begin{equation*}\label{apc1}
a_{n}(x)=\frac{\textit{a}(x)}{1+\frac{Q}{n} \textit{a}(x)}, \quad f_{n}(x)=\frac{f(x)}{1+\frac{1}{n}\vert f(x)\vert }\quad \text{and}\quad \mathcal{V}_n(x)=\frac{\mathcal{V}(x)}{1+\frac{1}{n}\vert \mathcal{V}(x)\vert }.
\end{equation*}

The Schauder theorem enables us to ensure the existence of a weak solution, i.e. $u_{n}\in W_{0}^{1,2}(\mathcal{O})$, for the approximate problem. Furthermore, this solution satisfying
\small{
\begin{equation}\label{wfa}
\begin{aligned}
\ds\int_{\mathcal{O}} \mathcal{M}(x) \nabla u_n  \nabla \varphi\;dx&+\int_{\mathcal{O}} a_n(x) h(u_n) \varphi\;dx=\int_{\mathcal{O}} T_n(u_n)\mathcal{V}_n(x) \nabla \varphi \;dx\\&+\lambda \int_{\mathcal{O}} \frac{T_n(u_n) \varphi}{\vert x\vert ^{2}+\frac{1}{n}}\;dx+\int_{\mathcal{O}} f_n(x) \varphi\;dx,
\end{aligned}
\end{equation}}
for every $\varphi$ in $W_{0}^{1,2}(\mathcal{O}) \cap L^{\infty}(\mathcal{O})$.

 Moreover, since for every fixed $n$ the functions  $\mathcal{V}_n$ and $f_n$ are bounded and since 
 \begin{equation*}
 \ds\left\vert\ds\lambda\frac{T_n(u_n)}{\vert x\vert ^{2}+\frac{1}{n}} \right\vert\leq \lambda n^2,
 \end{equation*}
 then $u_n$ belongs to $L^\infty(\mathcal{O})$, thanks to Stampacchia's boundedness theorem (see \cite{2}).

Next, we will prove the following Lemma by adapting the techniques developed in \cite{11}.
\begin{lem}
Let us assume \eqref{c1}, \eqref{c2}, \eqref{c44}-\eqref{Fhy} and $f \in L^{1}(\mathcal{O})$. If $\mathcal{A}$ and $\lambda$ are such that
 \begin{equation}
\alpha \mathcal{H}^2 >\mathcal{F}(m_{\lambda,\mathcal{A}}),
\end{equation}
 Then, there exist  two positive constants $\mathcal{C}_2 $ and  $\mathcal{C}_3 $ independent of $n$ and positive reel number $k_0$, such that
 \begin{equation}\label{Cpp2}
 \Vert u_n\Vert_{W^{1,2}_0(\mathcal{O})}\leq \mathcal{C}_2\Vert a\Vert_{L^{1}(\mathcal{O})}, \quad\quad \forall n\in\N,
 \end{equation}
 \begin{equation}\label{bnd}
 \left\Vert u_n \right\Vert_{L^\infty(\mathcal{O})}\leq k_0, \quad\quad \forall n\in\N,
 \end{equation}
 and 
 \begin{equation}\label{Cpp22}
 \Vert u_n\Vert_{L^{m}(\mathcal{O})}\leq \mathcal{C}_3\Vert a\Vert_{L^{1}(\mathcal{O})}, \quad\quad \forall n\in\N,
 \end{equation}
for every $ 2^*\leq m<m_{\lambda,\mathcal{A}}.$
\end{lem}
\begin{proof}
By utilizing $u_n$ as the test function in \eqref{wfa}, we derive
\begin{equation*}
\begin{aligned}
\int_{\mathcal{O}} \mathcal{M}(x) \nabla u_{n}  \nabla u_{n}\;dx&+\int_{\mathcal{O}} a_{n}(x) h(u_n)u_{n}\;dx=\int_{\mathcal{O}} T_n(u_n) \mathcal{V}_n(x)\nabla u_{n} \;dx\\&+\lambda\int_{\mathcal{O}} \frac{u_n T_n(u_n)}{\vert x\vert ^{2}+\frac{1}{n}}\;dx+\int_{\mathcal{O}} f_{n}(x) u_{n}\;dx,
\end{aligned}
\end{equation*}
From \eqref{c1}, \eqref{c2}, \eqref{c33}, \eqref{c44} and \eqref{gcdd}  it follows 
\begin{equation*}
\begin{aligned}
 \ds\alpha\int_{\mathcal{O}}\left\vert \nabla u_{n}\right\vert ^{2}\;dx +\int_{\mathcal{O}} a_{n}(x) h(u_n)u_{n}\;dx&\leq \mathcal{A}\int_{\mathcal{O}} \frac{\vert u_{n}\vert}{\vert x\vert }\vert \nabla u_{n}\vert \;dx\\&+ \lambda \int_{\mathcal{O}} \frac{u_{n}^2}{\vert x\vert ^{2}}\;dx+\int_{\mathcal{O}} \vert f_{n}(x)\vert\left\vert u_{n}\right\vert \;dx.
 \end{aligned}
\end{equation*}
Then, by Holder's and Hardy's inequalities,
 \small{
 \begin{equation}\label{est}
 \begin{aligned}
 \Big(\alpha-\frac{\lambda}{\mathcal{H}^2}-\frac{\mathcal{A}}{\mathcal{H}}\Big) &\int_{\mathcal{O}}\left\vert \nabla u_{n}\right\vert ^{2}+\int_{\mathcal{O}} a_{n}(x)\left(\vert h(u_n)\vert-Q\right)\vert u_{n}\vert\;dx\leq 0.
 \end{aligned}
 \end{equation}}
 Since $t\left( h(t)-Q\right)\geq -Qh^{-1}(Q)$ for every $0\leq t\leq h^{-1}(Q),$ it follows that
 \begin{equation}
 \begin{aligned}
 \left(\alpha-\frac{\lambda}{\mathcal{H}^2}-\frac{\mathcal{A}}{\mathcal{H}}\right) \int_{\mathcal{O}}\left\vert \nabla u_{n}\right\vert ^{2}\leq Qh^{-1}(Q)\int_{\mathcal{O}} \textit{a}(x)\,dx,
 \end{aligned}
 \end{equation}
 so that the sequence $\lbrace u_{n}\rbrace$ is bounded in $W^{1,2}_0(\mathcal{O}).$ 
 
 Note that, by \eqref{est},
 \begin{equation}
 \int_{\mathcal{O}} a_{n}(x)\left(\vert h(u_n)\vert-Q\right)\vert u_{n}\vert\;dx\leq 0,
 \end{equation}
 and thanks to  \eqref{gcdd} there exists $k_0$ such that $\vert h(s)\vert >Q$ for every $s>k_0$.
 
 Hence, recalling that $\left\vert G_{k_0}(u_n)\right\vert \leq \vert u_n \vert$, we get
 \begin{equation}
 \int_{\mathcal{O}} a_{n}(x)\left(\vert h(u_n)\vert-Q\right)\left\vert G_{k_0}(u_n)\right\vert\;dx\leq \int_{\mathcal{O}} a_{n}(x)\left(\vert h(u_n)\vert-Q\right)\vert u_{n}\vert\;dx\leq 0,
 \end{equation}
 which implies  that  \eqref{bnd} holds.
 

 In order to prove the summability result, we follow the idea of \cite{ABO}. Let $s>0$ and  consider the  function $\vert u_{n}\vert ^{s} u_{n}$ as a test function in the weak formulation \eqref{wfa}, we obtain
\begin{equation*}
\begin{aligned}
(s+1) &\int_{\mathcal{O}} \mathcal{M}(x) \nabla u_{n}  \nabla u_{n}\left\vert u_{n}\right\vert ^{s}\,dx+\int_{\mathcal{O}} a_{n}(x)h(u_n)\left\vert u_{n}\right\vert ^{s}u_n\,dx\\&=(s+1)\int_{\mathcal{O}} T_n(u_n) \mathcal{V}_n(x)\nabla u_{n}\left\vert u_{n}\right\vert ^{s}\,dx+\lambda\int_{\mathcal{O}} \frac{T_n(u_n) \left\vert u_{n}\right\vert ^{s}u_n}{\vert x\vert ^{2}+\frac{1}{n}}\,dx\\&+\int_{\mathcal{O}} f_{n}(x)\left\vert u_{n}\right\vert ^{s} u_{n}\,dx.
\end{aligned}
\end{equation*}
Putting to gather the hypotheses  \eqref{c1}, \eqref{c2}, \eqref{c33}, \eqref{c44} and \eqref{gcdd},  we obtain 
\small{
\begin{equation*}
\begin{aligned}
\alpha(s+1) &\int_{\mathcal{O}} \vert\nabla u_{n}\vert^2\left\vert u_{n}\right\vert ^{s}\,dx+\int_{\mathcal{O}} a_{n}(x)h(u_n)\left\vert u_{n}\right\vert ^{s}u_n\,dx\\&\leq \mathcal{A}(s+1)\int_{\mathcal{O}}\frac{\left\vert u_{n}\right\vert ^{s+1}}{\vert x \vert}\vert\nabla u_{n}\vert\,dx+\lambda\int_{\mathcal{O}} \frac{\left\vert u_{n}\right\vert ^{s+2}}{\vert x\vert ^{2}}\,dx+\int_{\mathcal{O}} \vert f_{n}(x)\vert\left\vert u_{n}\right\vert ^{s+1}\,dx.
\end{aligned}
\end{equation*}}
 For  $C_s=\ds\frac{4}{(s+2)^2}$, we have:
\begin{equation*}
\ds\vert\nabla u_{n}\vert^{2}\vert u_{n}\vert^{s}=C_s\left\vert\nabla\vert u_{n}\vert^{(\frac{s}{2}+1)}\right\vert^{2}.
\end{equation*}
Hence 
\small{
\begin{equation*}
\begin{aligned}
C_s\alpha(s+1) &\int_{\mathcal{O}} \left\vert\nabla\vert u_{n}\vert^{(\frac{s}{2}+1)}\right\vert^{2}\,dx+\int_{\mathcal{O}} a_{n}(x)h(u_n)\left\vert u_{n}\right\vert ^{s}u_n \,dx \\&\leq \frac{C_s^{\frac{1}{2}} \mathcal{A}(s+1)}{\mathcal{H}}\int_{\mathcal{O}}\left\vert\nabla\vert u_{n}\vert^{(\frac{s}{2}+1)}\right\vert^{2}\,dx+\frac{\lambda}{\mathcal{H}^2}\int_{\mathcal{O}} \left\vert\nabla\vert u_{n}\vert^{(\frac{s}{2}+1)}\right\vert^{2}\,dx\\&+\int_{\mathcal{O}} \vert f_{n}(x)\vert\left\vert u_{n}\right\vert ^{s+1}\,dx,
\end{aligned}
\end{equation*}}
at this point, we use \eqref{c44} and Sobolev inequality to obtain that
\small{\begin{equation}\label{inq1}
\begin{aligned}
  \frac{C(\alpha,s,\mathcal{A},\mathcal{H},\lambda)}{\mathcal{S}^2} \int_{\mathcal{O}} \vert u_{n}\vert^{2^*(\frac{s}{2}+1)}\,dx&\leq\int_{\mathcal{O}} a_{n}(x)(Q-\vert h(u_n)\vert )\left\vert u_{n}\right\vert ^{s+1}\,dx,\\
  &\leq \int_{\left\lbrace \vert h(u_n)\vert> Q \right\rbrace} a_{n}(x)(Q-\vert h(u_n)\vert )\left\vert u_{n}\right\vert ^{s+1}\,dx\\&+ \int_{\left\lbrace \vert h(u_n)\vert\leq Q \right\rbrace} a_{n}(x)(Q-\vert h(u_n)\vert )\left\vert u_{n}\right\vert ^{s+1}\,dx,
\end{aligned}
\end{equation} }
with $$C(\alpha,s,\mathcal{A},\mathcal{H},\lambda)=C_s\alpha(s+1) -\frac{C_s^{\frac{1}{2}} \mathcal{A}(s+1)}{\mathcal{H}}-\frac{\lambda}{\mathcal{H}^2}.$$
To handle the second term, on the right hand side, in \eqref{inq1} we use the following inequality
\begin{equation}
t^{s+1}\left( h(t)-Q\right)\geq -Q\left(h^{-1}(Q)\right)^{s+1}\quad \text{for every $0\leq t\leq h^{-1}(Q),$}
\end{equation}
and then, by making $m=\ds 2^*\left(\frac{s}{2}+1\right)$ and dropping the negative term in the right hand side, we arrive that
\begin{equation*}
\begin{aligned}
\frac{C(\alpha,s,\mathcal{A},\mathcal{H},\lambda)}{\mathcal{S}^2} &\left(\int_{\mathcal{O}} \vert u_{n}\vert^{m}\,dx\right)^{\frac{2}{2^*}}\leq Q\left(h^{-1}(Q) \right)^{\frac{m(N-2)}{N}-1}\int_{\left\lbrace \vert h(u_n)\vert\leq Q \right\rbrace} a_{n}(x)\,dx.
\end{aligned}
\end{equation*}
Let us observe that 
\begin{align*}
C(\alpha,s,\mathcal{A},\mathcal{H},\lambda)>0 &\Leftrightarrow \lambda+\mathcal{A}\mathcal{H}<\ds\alpha\frac{(N-2)^2(s+1)}{(s+2)^2}-\mathcal{A}\mathcal{H}\frac{s}{s+2},\\&\Leftrightarrow \lambda+\mathcal{A}\mathcal{H}<\mathcal{F}(m).
\end{align*}
By \eqref{Fhy}, there exist a unique value $m_{\lambda,\mathcal{A}}> 2^*$ such that $\mathcal{F}(m_{\lambda,\mathcal{A}})=\lambda+\mathcal{A}\mathcal{H}.$

 Thus, we have
 \begin{align*}
C(\alpha,s,\mathcal{A},\mathcal{H},\lambda)>0  \Leftrightarrow \mathcal{F}(m_{\lambda,\mathcal{A}})<\mathcal{F}(m).
\end{align*} 
From which
 \begin{equation*}
\begin{aligned}
 \int_{\mathcal{O}} \vert u_{n}\vert^{m}\,dx\leq C(\alpha,m,\mathcal{A},\mathcal{H},\lambda,Q)\left(\int_{\mathcal{O}} a(x)\,dx\right)^{\frac{2^*}{2}}\quad \text{for every $m\geq 2^*$,}
\end{aligned}
\end{equation*}
where $$C(\alpha,m,\mathcal{A},\mathcal{H},\lambda,Q)=\ds\left(\frac{Q\left(h^{-1}(Q)\right)^{\frac{m(N-2)}{N}-1}\mathcal{S}^2}{C(\alpha,s,\mathcal{A},\mathcal{H},\lambda)}\right)^{\frac{2^*}{2}}.$$
Thus, we conclude \eqref{Cpp22}. As a conclusion the proof of the Lemma is complete.
\end{proof}
\subsection{Proof of Theorem \ref{th2}}
In the previous lemma we have proved the boundedness of the sequence $\lbrace u_n \rbrace$ in $W^{1,2}_0(\mathcal{O})$. Therefore, up to subsequence, there exists a function $u\in W^{1,2}_0(\mathcal{O})$ such that 
\begin{align}
&u_n\rightarrow u\text{ a.e. in }\mathcal{O},\label{cvaeu}\\
&u_n\rightharpoonup u \text{ weakly in $W^{1,2}_0(\mathcal{O})$},\label{cvfau}\\
&u_n\rightarrow u \text{ strongly in $L^{2}(\mathcal{O})$}.\label{cvstu}
\end{align}
Using that $$\ds\left\vert \textit{a}_{n}(x) h\left(u_{n}\right)\right\vert \leq \textit{a}(x) \max _{\vert s\vert \leq k_{0}}\vert h(s)\vert $$
and by Lebesgue theorem, we deduce that
\begin{equation*}
\ds\left\{a_{n}(x) h\left(u_{n}\right)\right\}\rightarrow \ds\textit{a}(x) h(u)\quad\text{strongly in } L^1(\mathcal{O}).
\end{equation*}
Moreover, we obtain from  \eqref{Cpp2} and \eqref{cvaeu}  that
\begin{equation*}
T_n(u_n)\mathcal{V}_n(x)\rightarrow u\mathcal{V}(x)\quad\text{strongly in } (L^2(\mathcal{O}))^2.
\end{equation*}
In addition, we have
\begin{equation*}
\ds\frac{\lambda}{\vert x\vert^2+\frac{1}{n}}\rightarrow \frac{\lambda}{\vert x\vert^2}\quad \text{strongly in } L^t(\mathcal{O}),\; \text{for every}\; 1\leq t<\frac{N}{2}.
\end{equation*}
Combining all the previous convergences and \eqref{cvfau} one can pass to the limit in the weak formulation \eqref{wfa} to obtain that the limit $u$ satisfies \eqref{wf}.

Let $u$ and $v$ be two weak solution of \eqref{prb1}. Using $u-v$ as test function in \eqref{wf} to get
\small{
\begin{equation*}
\begin{aligned}
&\ds\int_{\mathcal{O}} \mathcal{M}(x) \nabla(u-v)  \nabla (u-v)\;dx+\int_{\mathcal{O}} \textit{a}(x)(h(u)-h(v)) (u-v)\;dx\\&=\int_{\mathcal{O}} (u-v)\mathcal{V}(x) \nabla (u-v) \;dx+\lambda \int_{\mathcal{O}} \frac{(u-v)^2 }{\vert x\vert ^{2}}\;dx.
\end{aligned}
\end{equation*}}
Note $w=u-v$ and thanks to \eqref{c1}-\eqref{c44} we get
\begin{equation*}
\begin{aligned}
\ds\alpha\int_{\mathcal{O}}\vert\nabla w\vert^2\;dx+\int_{\mathcal{O}} \textit{a}(x)(h(u)-h(v)) w\;dx&\leq \mathcal{A}\int_{\mathcal{O}} \frac{\vert w \vert}{\vert x \vert} \vert\nabla w\vert \;dx+\lambda \int_{\mathcal{O}} \frac{w^2}{\vert x\vert ^{2}}\;dx,
\end{aligned}
\end{equation*}
using \eqref{c2}, Hardy's inequality and dropping the positive term we obtain that
\begin{equation*}
\begin{aligned}
\ds\left(\alpha-\frac{\mathcal{A}}{\mathcal{H}}-\frac{\lambda}{\mathcal{H}^2}\right)\int_{\mathcal{O}} \vert\nabla w\vert^2\;dx&\leq 0,
\end{aligned}
\end{equation*} 
which implies, from \eqref{H1}, that $u=v.$

\subsection{Example}
In order to prove the optimality of the condition \eqref{H1} we work as in \cite{ABO}, it results that 
for any $\ds\rho <1+\frac{N}{2}$ and $\mathcal{V}(x)=\ds\frac{x}{\vert x\vert^2}$,  the function $u_\rho(x)=\ds\frac{1}{\vert x\vert^\rho}-1$ is a weak solution in $ W^{1,2}_0(\mathcal{B})$, with $\mathcal{B}=\lbrace x\in \R^N: \vert x\vert<1 \rbrace$, of the following problem 
\begin{equation}\label{kl}
\begin{cases}
\begin{aligned}
-\Delta u_\rho&= \text{div}(\rho u_\rho\mathcal{V}(x))+\rho\frac{N-2}{\vert x\vert^2}  & & \text { in } \mathcal{B}, \\
u_\rho &=0 & & \text { on } \partial \mathcal{B} .
\end{aligned}
\end{cases}
\end{equation}
Let $\alpha>0,$ $\mathcal{A}=1$, $\rho=\ds\frac{N}{m},$  $m_{\lambda,\mathcal{A}},$ $m$ and $\lambda$  such that $\mathcal{F}(m)<\mathcal{A}\mathcal{H}+\lambda.$

 Being $\mathcal{V}(x)=\ds\frac{x}{\vert x\vert^2}$, \eqref{kl} with the following formulas  
 \begin{equation*}
\text{div}(u_\rho\mathcal{V}(x))=\frac{N-(\rho+2)}{\vert x\vert^{\rho+2}}-\frac{N-2}{\vert x\vert^{2}},
\end{equation*}
\begin{equation*}
\text{div}(\nabla u_\rho)=-\frac{1}{\vert x\vert^{\rho+2}}\left( N\rho -\rho(\rho+2)\right),
\end{equation*}
give that $u$ is the solution, in the weak sense, of the following Dirichlet problem
\begin{equation}\label{sz}
\begin{cases}
\begin{aligned}
-\Delta u_\rho+a_m(x)h(u_\rho) &=-\text{div}(u_\rho\mathcal{V}(x))+\lambda \frac{u_\rho}{\vert x \vert^2}+f_m(x) & & \text { in } \mathcal{B}, \\
u_\rho &=0 & & \text { on } \partial \mathcal{B} .
\end{aligned}
\end{cases}
\end{equation}
with
\begin{equation*}
\begin{aligned}
C(\alpha,\rho)&=\ds\rho(\alpha+1)(N-\rho-2)+\frac{\rho}{2}-\frac{3}{2}(N-2),\\
 a_m(x)&=\ds\frac{\lambda-C(\alpha,\rho)+\mathcal{F}(m)}{\vert x\vert^2}, \\
   Q_m&\geq\ds\frac{\mathcal{F}(m)+C(\alpha,\rho)-N+2}{\lambda-C(\alpha,\rho)+\mathcal{F}(m)},\\ 
   f_m(x)&=\ds\frac{\mathcal{F}(m)+C(\alpha,\rho)-N+2}{\vert x\vert ^2},\\
   h(u_\rho)&= u_\rho.
\end{aligned}
\end{equation*}
We stress that $\ds\frac{\mathcal{F}(m)+C(\alpha,\rho)-N+2}{\lambda-C(\alpha,\rho)+\mathcal{F}(m)}>0$  being $\mathcal{F}(m)<\mathcal{A}\mathcal{H}+\lambda$. Now let us  consider the case where $\lambda=1,$ $\alpha=7$ and $N=3$, the following Figure below gives the region where the  regularity result, of the solution $u_\rho$, is hold with respect the variation of the function $\mathcal{F}$.

	\begin{figure}[H]
\begin{center}
 \includegraphics[width=0.7\linewidth]{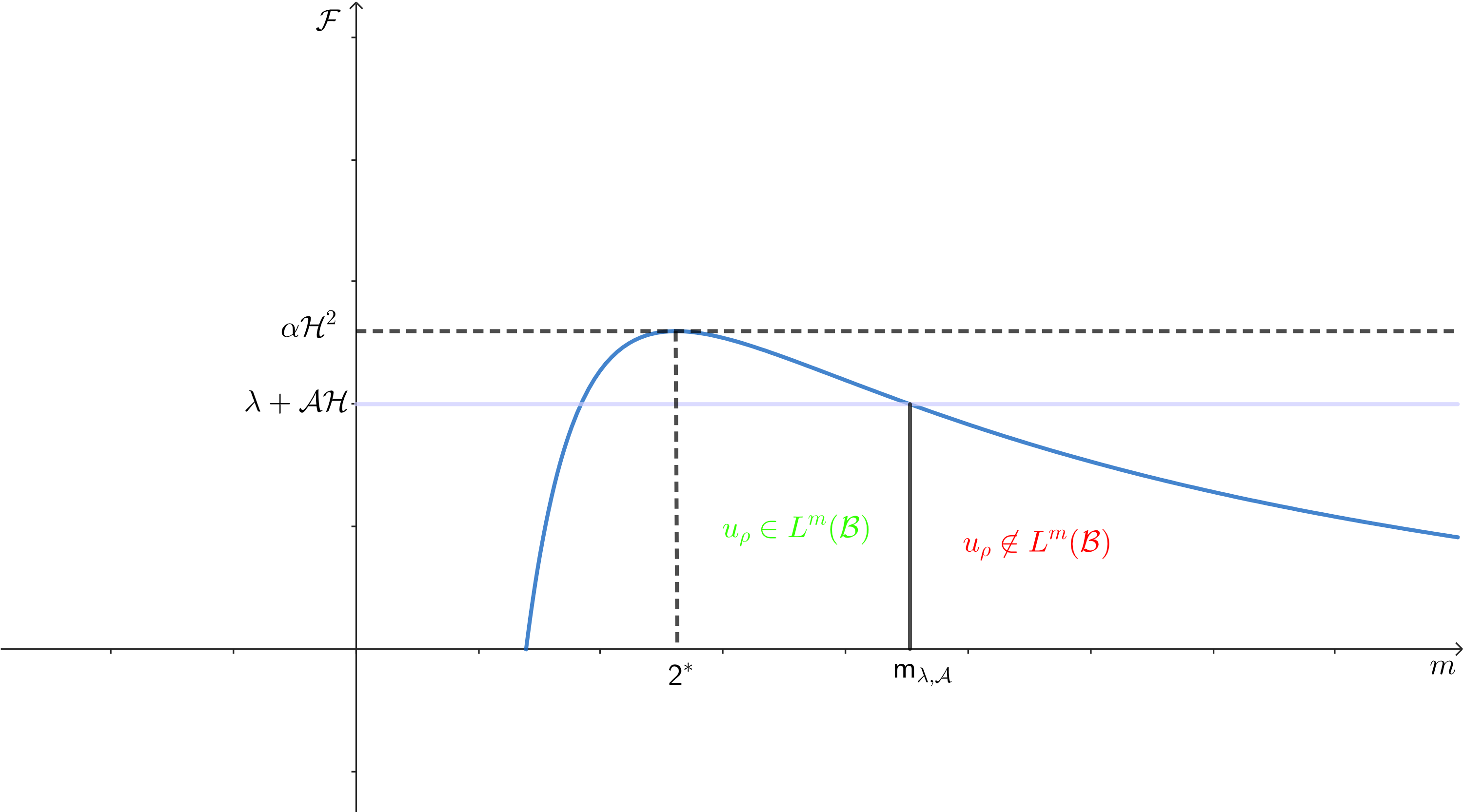}
		\caption{Graph of the function $\mathcal{F}(m)=\frac{21}{m}(1-\frac{3}{m})+\frac{1}{2}(\frac{6}{m}-1)$.}
 \end{center}
 \end{figure}
%
According to the above definitions we have that $\eqref{c1},\eqref{c2}$ and $\eqref{c33}-\eqref{gcdd}$ are hold and the solution $u_\rho$ is a weak solution of \eqref{sz}.  Nevertheless, for every $m>m_{\lambda, \mathcal{A}}$, $u_\rho$ does not belong to $L^m(\mathcal{B})$. This prove the optimality of \eqref{H1}.
\subsection{Proof of the Theorem \ref{wmp} }
 Let us choose $-T_{\kappa}\left(u^{-}\right)$ as test function in \eqref{wf}. Thus 
 \begin{equation*}
 \begin{aligned}
 \int_{\mathcal{O}}& \mathcal{M}(x)\nabla u \nabla T_{\kappa}(u^{-})\;dx-\int_{\mathcal{O}} \textit{a}(x)h(u)T_{\kappa}(u^{-})\;dx=-\int_{\mathcal{O}} u \mathcal{V}(x) \nabla T_{\kappa}(u^{-})\;dx\\&-\lambda\int_{\mathcal{O}} \frac{u T_\kappa(u^-)}{\vert x\vert^2}\;dx-\int_{\mathcal{O}} f T_{\kappa}(u^{-})\;dx,
 \end{aligned}
 \end{equation*}
thanks to \eqref{c1}-\eqref{gcdd} and \eqref{hy}, we have 
\begin{equation*}
\begin{aligned}
\alpha \int_{\mathcal{O}}\left\vert  \nabla T_{\kappa}(u^{-})\right \vert ^{2}\;dx &\leq \mathcal{A}\int_{ \mathcal{O}} \frac{\left\vert u \nabla T_\kappa(u^{-})\right\vert}{\vert x \vert}\;dx+\lambda\int_{\mathcal{O}} \frac{u^- T_\kappa(u^-) dx}{\vert x\vert^2}\;dx\\&+\kappa\int_{\mathcal{O}} \textit{a}(x)\left(h(u)-Q\right)\;dx,
\end{aligned}
\end{equation*}
Observing that
\begin{align*}
\begin{cases}
\ds\mathcal{A}\int_{ \mathcal{O}} \frac{\left\vert u \nabla T_\kappa(u^{-})\right\vert}{\vert x \vert}\;dx\leq\mathcal{A}\int_{ \mathcal{O}} \frac{\left\vert u^+ \nabla T_\kappa(u^{-})\right\vert}{\vert x \vert}\;dx+\mathcal{A}\int_{ \mathcal{O}} \frac{\left\vert u^- \nabla T_\kappa(u^{-})\right\vert}{\vert x \vert}\;dx\\
\\
\ds \lambda\int_{\mathcal{O}} \frac{u T_\kappa(u^-)}{\vert x\vert^2}\;dx=\lambda\int_{\mathcal{O}} \frac{u^{+} T_\kappa(u^-)}{\vert x\vert^2}\;dx+\lambda\int_{\mathcal{O}} \frac{u^- T_\kappa(u^-)}{\vert x\vert^2}\;dx\\
\\
 \ds\kappa\int_{\mathcal{O}} \textit{a}(x)\left(h(u)-Q\right))\;dx\leq \kappa\int_{\mathcal{O}} \textit{a}(x)\left(\max_{\vert s\vert\leq k_0}h(s)-Q\right)\;dx
\end{cases}
\end{align*}
and
\begin{align*}
\begin{cases}
t^+ T_\kappa(t^-)=0,\quad \forall t\in\R\\
\\
t^+ T^\prime_\kappa(t^-)=0,\quad \forall t\in\R.
\end{cases}
\end{align*}
Thus, we obtain
\begin{equation*}
\begin{aligned}
\alpha \int_{\mathcal{O}}\left\vert  \nabla T_{\kappa}(u^{-})\right \vert ^{2}\;dx &\leq \mathcal{A}\int_{ \mathcal{O}} \frac{\left\vert T_\kappa(u^-)\right\vert}{\vert x \vert}\left\vert \nabla T_\kappa(u^{-})\right\vert\;dx+\lambda\int_{\mathcal{O}} \frac{u^- T_\kappa(u^-) dx}{\vert x\vert^2}\;dx\\&+\kappa\int_{\mathcal{O}} \textit{a}(x)\left(\max_{\vert s\vert\leq k_0}h(s)-Q\right)\;dx,
\end{aligned}
\end{equation*}
Applying Holder's and Young's's inequalities and \eqref{hky1}, we get
\begin{equation*}
\begin{aligned}
\left(\alpha-\frac{\mathcal{A}}{\mathcal{H}}-\frac{\lambda}{\mathcal{H}^2}\right)\int_{\mathcal{O}}\left\vert  \nabla T_{\kappa}(u^{-})\right \vert ^{2}\;dx &\leq \frac{\lambda k_0}{4}\int_{\lbrace -k<u<0\rbrace} \frac{ dx}{\vert x\vert^2}\;dx,
\end{aligned}
\end{equation*}
Let  $0<\kappa<\delta$ and following the argument contained in  \cite{13}, we can say that  the previous inequality imply
\begin{equation*}
meas\lbrace u<-\delta\rbrace\leq C(\alpha,k_0,\mathcal{S},\mathcal{A})\int_{\lbrace -\kappa<u<0 \rbrace} \frac{dx}{\vert x \vert^2},
\end{equation*}
since $\ds\frac{1}{\vert x \vert^2}  \in L^{1}(\mathcal{O})$, the right hand side goes to 0 , as $\kappa \rightarrow 0$. 

In such a way, we deduce that
\begin{equation*}
meas \lbrace u<-\delta \rbrace=0,\quad \text{ for every $\delta>0$.}
\end{equation*}
\subsection{Proof of the Theorem \ref{smp}}
  Considering the real valued function $\psi: \R\rightarrow \R$ defined by
 $$ \psi(t)=\frac{1}{l+t},\quad l \in \left]1,+\infty\right[.$$
   Using
 \begin{equation*}
\begin{aligned}
v=\psi(u)\phi^{2}, \quad 0 \leq \phi \leq 1,\;\; \phi \in W_{0}^{1,2}(\mathcal{O}) \cap L^{\infty}(\mathcal{O}),\quad 
\end{aligned}
\end{equation*}
as test function in \eqref{wfa}. Then
\begin{equation*}
\begin{aligned}
&\int_{\mathcal{O}} \mathcal{M}(x) \nabla u \nabla u \psi^\prime(u)\phi^2\,dx+2 \int_{\mathcal{O}} \mathcal{M}(x) \nabla u \psi(u)\phi\nabla \phi \,dx+\int_\mathcal{O} \textit{a}(x)h(u)\psi(u)\phi^2\,dx\\&= \int_{\mathcal{O}} u\mathcal{V}(x)\nabla u \psi^\prime(u)\phi^2\,dx
+2\int_{\mathcal{O}} u\mathcal{V}(x)\psi(u)\phi \nabla \phi\,dx
+ \lambda\int_{\mathcal{O}} \frac{u\psi(u)\phi^2}{\vert x\vert^2}\,dx\\&+\int_{\mathcal{O}} f\psi(u)\phi^2\,dx,
\end{aligned}
\end{equation*}
and by \eqref{c1}, \eqref{c2}, Holder's, Hardy's and Young's inequalities we obtain
\begin{equation*}
\begin{aligned}
\frac{\alpha}{2}&\int_\mathcal{O} \frac{\vert \nabla u\vert^2}{(l+u)^2}\phi^2\,dx+\lambda\int_{\mathcal{O}} \frac{u\phi^2}{(l+u)\vert x\vert^2}\,dx+\int_{\mathcal{O}} f\frac{\phi^2}{l+u}\,dx\\&\leq
\left(\frac{4\beta^2}{\alpha}+\frac{\mathcal{A}^2}{\mathcal{H}^2\alpha}+\frac{2\mathcal{A}}{\mathcal{H}}\right)\int_\mathcal{O} \vert \nabla \phi \vert^2\,dx+\int_\mathcal{O} \textit{a}(x)h(u)\frac{\phi^2}{l+u}\,dx,
\end{aligned}
\end{equation*}
since we suppose that $f(x)\geq 0$ then we have $u\geq 0$ as a consequence of the Theorem \ref{wmp}, which implies that 
\begin{equation*}
\begin{aligned}
\frac{\alpha}{2}\int_\mathcal{O} \frac{\vert \nabla u\vert^2}{(l+u)^2}\phi^2\,dx&\leq
C\left(\alpha,\beta,\mathcal{A},\mathcal{H}\right)\int_\mathcal{O} \vert \nabla \phi \vert^2\,dx+\int_\mathcal{O} \textit{a}(x)h(u)\frac{\phi^2}{l+u}\,dx,
\end{aligned}
\end{equation*}
hence, using the fact that $\Vert u\Vert_{L^\infty(\mathcal{O})}\leq k_0$ gives
\begin{equation*}
\begin{aligned}
\frac{\alpha}{2}\int_\mathcal{O} \left\vert \nabla \log\left(1+\frac{u}{l}\right) \right\vert^2\phi^2\,dx&\leq
C\left(\alpha,\beta,\mathcal{A},\mathcal{H}\right)\int_\mathcal{O} \vert \nabla \phi \vert^2\,dx+C_{k_0}\int_\mathcal{O} \textit{a}(x)\frac{\phi^2}{l+u}\,dx.
\end{aligned}
\end{equation*}
Now, we reason by contradiction (see \cite{11} for instance) in order 
 to establish that $u=0$ at most on a set of zero measure. Let $Z:=\{x \in \mathcal{O}: u(x)=0\}$ and assuming  that  $meas(Z)\geq 0$ . Let $\omega \subset \subset \mathcal{O}$ be an open set such that $Z \cap \omega$ has positive measure. Choosing in the above inequality a function $\phi \in W_{0}^{1,2}(\mathcal{O}) \cap L^{\infty}(\mathcal{O})$ such that $\phi \equiv 1$ in $\omega$, we obtain
\begin{equation*}
\begin{aligned}
\frac{\alpha}{2}\int_\omega \left\vert \nabla \log\left(1+\frac{u}{l}\right) \right\vert^2\,dx&\leq
C\left(\alpha,\beta,\mathcal{A},\mathcal{H}\right)\int_\mathcal{O} \vert \nabla \phi \vert^2\,dx+C_{k_0}\int_\mathcal{O} \textit{a}(x)\frac{dx}{l+u},\\
&\leq C\left(\alpha,\beta,\mathcal{A},\mathcal{H}\right)\int_\mathcal{O} \vert \nabla \phi \vert^2\,dx+C_{k_0}\int_\mathcal{O} \textit{a}(x)\,dx\\&\leq C+C_{k_0}\int_\mathcal{O} \textit{a}(x)\,dx.
\end{aligned}
\end{equation*}
Being $u=0$ in the subset $Z \cap \omega$, we can apply the Poincare inequality to derive that
\begin{equation*}
\begin{aligned}
\ds\int_\omega \left\vert \nabla \log\left(1+\frac{u}{l}\right) \right\vert^2\,dx& \leq C_1 + C_2\int_\mathcal{O} \textit{a}(x)\,dx,
\end{aligned}
\end{equation*}
 with  $C_1$ and $C_2$ are positive constants independent of $l$.
 
Since we have for every $\epsilon \geq 0$
\begin{equation*}
\left(\log \left(1+\frac{\epsilon}{l}\right)\right)^{2} \text { meas }\left(\{x \in \omega: u(x)>\epsilon\}\right) \leq \int_{\omega}\left(\log \left(1+\frac{u}{l}\right)\right)^{2}\,dx,
\end{equation*}
then, we  conclude that
\begin{equation*}
\ds meas\left(\{x \in \omega: u(x)>\epsilon\}\right) \leq \lim _{h \rightarrow 0} \frac{C_1 + C_2\ds\int_\mathcal{O} \textit{a}(x)\,dx}{\left(\log \left(1+\frac{\epsilon}{l}\right)\right)^{2}},
\end{equation*}
which implies 
\begin{equation*}
meas\left(\{x \in \omega: u(x)>\epsilon\}\right)=0 \quad \quad \text{for every $\epsilon \geq 0$},
\end{equation*}
or equivalently that
\begin{equation*}
u \equiv 0 \quad \quad \text{in}\; \omega,
\end{equation*}
 then the arbitrariness of the open set $\omega \subset \subset \mathcal{O}$ such that $Z \cap \omega$ has positive measure implies then that  $f(x)$ should be equal to
zero a.e. in $\omega$, which is a contradiction. Thus, there exist a strictly positive constant $m_\omega$ such that $u(x)\geq m_\omega>0$ a.e. in $\omega,$ and the proof is complete.
\section*{Declarations}
\subsection*{Data Availability}
Our manuscript has no associate data.
\subsection*{Conflict of interest}
The author declare that he has no conflict of interest.
\subsection*{Funding}
 This work was supported by Ministero degli Affari Esteri e della Cooperazione Internazionale under Grant Agreement No: 1723.
\subsection*{Authors' contributions}
 These authors contributed equally to this work.
 \subsection*{Acknowledgement(s)}
This work has been done while the first author was  PhD  visiting student at the Department of Mathematics and Computer Science, University of Catania, supported by a  MAECI scholarship of the Italian Government.

\end{document}